\documentclass[12pt]{article}

\usepackage{amsfonts}
\usepackage{graphics}
\usepackage{amssymb}

\newtheorem{theorem}{Theorem}[section]

\newtheorem{definition}{Definition}[section]
\newtheorem{proposition}[theorem]{Proposition}

\newtheorem{remark}{Remark}[section]

\begin{document}

\begin{center}
{\Large   Stochastic integration with respect to the cylindrical Wiener process via regularization  \\}
\end{center}

\vspace{0.3cm}

\begin{center}

{\large
Christian Olivera\footnote{Research  supported by CAPES PNPD N02885/09-3.}}\\
\textit{Departamento de Matem\'{a}tica- Universidade Federal de S\~ao Carlos
Rod. Washington Luis, Km 235 - C.P. 676 - 13565-905 S\~ao Carlos, SP - Brasil
\\ Fone: 55(016)3351-8218
\\  e-mail:  colivera@dm.ufscar.br}
\end{center}

\vspace{0.3cm}

\begin{center}
\begin{abstract}
\noindent Following the ideas of F. Russo and P. Vallois we use the notion of
forward    integral to introduce a new stochastic integral respect to
the cylindrical   Winer process.  This integral is an extension of the classical  integral. As an application, we prove existence  of
 solution of a parabolic stochastic differential partial equation with anticipating stochastic initial  date.
\end{abstract}
\end{center}


\vspace{0.3cm} \noindent {\bf MSC2000 subject classification:}
60H05 , 60H15.

\section {Introduction}

Integration via regularization  was introduced by F. Russo and P. Vallois
 ( see \cite{RV1}, \cite{RV5} and  \cite{RV3}) and it have been studied and developed by
many authors. They introduced  the forward
(resp. symmetric) integral $\int_0^T Y d^-X$ (resp. $\int_0^T Y
d^\circ X$),  the first integral is a
generalization of It\^o's integral, the second one is an extension
of Stratonovich integral, when the integrator $X$ is a
semimartingale and the integrand $Y$ is predictable. Moreover,  this approach is
connected with the Colombeau theory of generalized functions(see \cite{Russo}). A recent survey on the subject is \cite{RVSem}.

\noindent  In this work we take the approach of regularization to define a new integral
 respect to the cylindrical   Wiener process.  This integral is an extension of the classical one.
 Our method by define the integral via regularization   is
based in the following representation of the  stochastic integral respect to the
the cylindrical  Winer process

\[
\int_{0}^{T}  g_{s} dB_{s} = \sum_{j=1}^{\infty} \int_{0}^{T}  <g, v_{j}>_{V_{Q}} dB(v_{j}),
\]

 \noindent see sections 3 for details.  As an application
 we  study the  following  SPDE,

\begin{equation}\label{intro}
 \left \{
\begin{array}{lll}
u \ dt & = & \triangle u \ dt + g(t,x,u)  dW_{t}, \   \\
u(t,0) & = & u(t,1)=0, \   \\
u_0 & = & f(x,F)
\end{array}
\right .
\end{equation}

\noindent where  $F$ be a   $ \mathcal{F}_T$-measure d-dimensional  r.v and $W$ is  a cylindrical  Winer process in   $L^{2}([0,1])$.
  There is great interest in solving the  problem
(\ref{intro}) (see for instance \cite{Da}, \cite{LORo}, \cite{Nual} ,
 \cite{Tindel}  and \cite{Wals} and references). In this context, and under
general assumptions, we show existence  for the
SPDE (\ref{intro}). The authors do not know any work covering this situation in the literature.

\noindent The presentation is organized as follows: in Section 2, we
provide some basic concepts on  the theory of stochastic integrals with
respect to the cylindrical Wiener processes and on stochastic calculus via regularization.
In section 3, we  introduced a new definition of integral  respect to the cylindrical Wiener processes  and
we show that it is an extension of the  integral defined in section 2 . In Section 4,
we study the  SPDE (\ref{intro}) in the context of this new integration. Moreover, some
future works are indicated.

\section{ Preliminaries.}

\subsection{Stochastic integration with respect to the cylindrical Wiener process}

 \noindent Fix a separable Hilbert space $V$ with inner product $<. ,.> $.  Following \cite{Metr} and
\cite{Da}, we define the general
notion of  cylindrical Wiener process in V.

\begin{definition}
Let Q be a symmetric (self-adjoint) and non-negative definite bounded linear operator
on V. A family of random variables $B = \{B_t(h), t \geq 0, h \in V \}$ is a cylindrical Wiener process on V if the
following two conditions are fulfilled:
\begin{enumerate}
\item for any $h \in V$, $\{B_t (h), t \geq 0\}$ defines a Brownian motion with variance $t<Qh, h>$ ,

\item for all $s, t \in R+$ and $ h, g \in V$

\[
\mathbb{E}(B_{s}(h)(B_{t}(g))=(s \wedge t) <Qh, g>
\]
\end{enumerate}
\noindent where $s \wedge t := min (s, t)$. If $Q = Id_V$ is the identity operator in V, then B will be called a
standard cylindrical Wiener process. We will refer to $Q$ as the covariance of $B$.
\end{definition}

\noindent Let  $\mathcal{F}_t$ be the $\sigma$-field generated by the random variables ${B_{s}(h), h\in V, 0  \leq \leq t}$ and the P-null
sets. We define the predictable $\sigma$-field as the $\sigma$-field in $[0, T ] \times \Omega$  generated by the sets ${(s, t]\times A, A\in\mathcal{F}_s,  0 \leq s < t \leq T }$.

\noindent  We denote by $V_Q$ (the completion of) the Hilbert space V endowed with the inner-product

\[
<h, g>_{V_{Q}}
:=<Qh, g> ,\ \  h, g \in V.
\]

\noindent We can now define the stochastic integral of any predictable square-integrable process with values in
$V_Q$ , as follows. Let $\{v_j , j\in \mathbb{N} \}$ be a complete orthonormal basis of the Hilbert space $V_Q$ . For any predictable
process $g \in L^{2}(\Omega \times [0, T ],V_Q )$ it turns out that the following series is convergent in $L^{2}(\Omega, \mathcal{F}, P)$ and
the sum does not depend on the chosen orthonormal system:

\begin{equation}\label{inteCy}
g\cdot B := \sum_{j=1}^{\infty} \int_{0}^{T}  <g, v_{j}>_{V_{Q}} dB(v_{j}).
\end{equation}

\noindent  We notice that each summand in the above series is a classical Ito integral with respect to a standard
Brownian motion, and the resulting stochastic integral is a real-valued random variable. The stochastic
integral $g\cdot B $ is also denoted by $\int_{0}^{T}  g_{s} dB_{s} $ . The independence of the terms in the series (\ref{inteCy}) leads to
the isometry property

\[
  \mathbb{E}((g\cdot B)^{2})= \mathbb{E}(( \int_{0}^{T}  g_{s} dB_{s} )^{2})
   = \mathbb{E}( \int_{0}^{T}  \|g\|_{V_{Q}}^{2} ds).
\]

\subsection{Russo-Vallois Integrals }

\noindent Let $(X_t)_{ t\geq 0} $ be a continuous process and
$(Y_t)_{ t\geq 0}$ be a process with paths in
$L_{loc}^{1}(\mathbb{R}^{+})$, i.e. for any $ b > 0$, $
\int_{o}^{b}|Yt| dt <\infty$ ‡ a.s. The generalized stochastic
integrals and covariations are defined through a regularization
procedure. That is, let $I^{-}(\epsilon, Y, dX)$ (resp.
$I^{+}(\epsilon, Y, dX), I^{0}(\epsilon, Y, dX)$ and $C(\epsilon,
Y,X))$) be the $\epsilon-$forward integral (resp.
$\epsilon-$backward integral, $\epsilon-$symmetric integral and
$\epsilon-$covariation):

$$
I^{-}(\epsilon, Y, dX)=\int_{0}^{t} Y_{s}
\frac{(X_{s+\epsilon}-X_{s})}{\epsilon} ds \ t \geq 0,
$$

$$
I^{+}(\epsilon, Y, dX)=\int_{0}^{t} Y_{s}
\frac{(X_{s}-X_{s-\epsilon})}{\epsilon} ds \ t \geq 0,
$$

$$
I^{0}(\epsilon, Y, dX)=\int_{0}^{t} Y_{s}
\frac{(X_{s+\epsilon}-X_{s-\epsilon})}{2\epsilon} ds \ t \geq 0,
$$

$$
C(\epsilon, Y, X)=\int_{0}^{t} \frac {(Y_{s+\epsilon}-Y_{s} )
(X_{s+\epsilon}-X_{s})}{\epsilon} ds \ t \geq 0,
$$

\begin{definition}\label{defi} We set

$$
Forward \ integral  \ \int_{0}^{t} Y dX^{-}=
\lim_{\epsilon\rightarrow 0}I^{-}(\epsilon, Y, dX)(t),
$$

$$
Backward \ integral  \ \int_{0}^{t} Y dX^{+}=
\lim_{\epsilon\rightarrow 0}I^{+}(\epsilon, Y, dX)(t),
$$

$$
Symmetric \ integral  \ \int_{0}^{t} Y dX^{0}=
\lim_{\epsilon\rightarrow 0}I^{0}(\epsilon, Y, dX)(t),
$$

$$
Covariation \ [X, Y ]_t = \lim_{\epsilon\rightarrow 0}C(\epsilon, Y,
X)(t),
$$

\noindent provided the limit exist ucp.
\end{definition}

\begin{remark}

\begin{enumerate}

\item[a)]  The links between $[X,Y]$, backward, symmetric and  forward  integrals
are the following

$$
\int_{0}^{t} Y dX^{\circ}= \int_{0}^{t} Y dX^{-} + \frac{1}{2} [X,Y](t).
$$

$$
[X,Y](t)=\int_{0}^{t} Y dX^{+} - \int_{0}^{t} Y dX^{-}.
$$

 \item[b)] If $X$ has locally bounded variation and $Y$ is
cad lag then $ \int_{0}^{t} Y dX^{-}=\int_{0}^{t} Y dX^{+}$ and it
is the pathwise Stieljes integral  (see \cite{RV1}).

\item[c)] Let X  and Y be continuous  semimartingales,  H  be an adapted process.  We denote by
$\int_{0}^{t} H dX$ the It\^o integral. Therefore, we have (see
\cite{RV1} )

$$
\int_{0}^{t} H dX^{-}=\int_{0}^{t} H dX,
$$
\noindent and

$$
[X, Y ]_t \ is \ the \ usual \ covariation \ of \ X \ and \ Y .
$$

\end{enumerate}

\end{remark}

\noindent For a recent account in the subject
we refer the reader to   \cite{RVSem}.

\section{Stochastic integration with respect to the cylindrical Wiener process via regularization }

\noindent  In this  section, strongly inspired in the  integrals via regularization,  we define a new stochastic  integral with respect to
 the cylindrical Wiener process.  We will investigate the link between this new integral
   and the integration defined in section 2.

We shall begin giving the following definition

\begin{definition}\label{intreRusso} Let $g$ be a $V_{Q}$-valued  stochastic process  such that the Fourier coefficient  $<g, v_{j}>_{V_{Q}}$ are finite for all $t$ and $\omega$. Suppose that for all $j\in \mathbb{N}$ there exists
\[
c_j =\int_{0}^{T} <g, v_{j}>_{V_Q} dB_{s}^{-}(v_{j})
\]
and that    $\lim_{m\rightarrow\infty } \sum_{j=1}^{m} c_j $ converges in probability.  We define the {\it{Russo-Vallois-forward  \ integral
}} $ \int_{0}^{T}  g_{s} dB_{s}^{-}$ by

\begin{equation}\label{pro1}
\int_{0}^{T}  g_{s} dB_{s}^{-}:= \lim_{m\rightarrow\infty } \sum_{j=1}^{m} c_j .
\end{equation}

\end{definition}

\begin{remark}
\begin{enumerate}

 \item
 It is clear of the definition that the Russo-Vallois-forward  \ integral is a linear  operator.

\item

 We note  that not require   any condition of  adaptability on the process $g$
  through the definition   \ref{intreRusso}.

\item

 In the case that $V=\mathbb{R}$ and $Q=Id_{\mathbb{R}}$ the Russo-Vallois-forward
 integral is equal to forward   integral in the sense of the definition \ref{defi}.

 \end{enumerate}

\end{remark}

\begin{proposition}  For any predictable
process $g \in L^{2}(\Omega \times [0, T ],V_Q )$ the  integral (\ref{inteCy}) is equal to the
 Russo-Vallois-forward integral .
\end{proposition}

\begin{proof}  Since  $<g, v_{j}>_{V_Q}$ is predictable  and belongs to $ L^{2}(\Omega \times [0, T ])$ we have that

 \[
c_j =\int_{0}^{T} <g, v_{j}>_{V_Q} dB_{s}^{-}(v_{j})=\int_{0}^{T} <g, v_{j}>_{V_Q} dB_{s}(v_{j}).
\]

\noindent From the $L^{2}(\Omega)$-convergence of the series

\[
 \int_{0}^{T}  g_{s} dB_{s}= \sum_{j=1}^{\infty} \int_{0}^{T}  <g, v_{j}>_{V_{Q}} dB(v_{j}),
\]

\noindent it follows that
\[
\int_{0}^{T}  g_{s} dB_{s}^{-}=\int_{0}^{T}  g_{s} dB_{s}.
\]
\end{proof}

\begin{remark} It is easy to see examples of stochastic process  that are
Russo-Vallois-forward integral but are not integrable  It\^o . See remark 14 in
\cite{RVSem}.
\end{remark}

\section{Stochastic parabolic equation.}

\noindent The stochastic partial differential equation type  (\ref{para}) can be analyzed by different approach
related with classical deterministic methods, see for example  \cite{Da}, \cite{LORo},  \cite{Nual},
 \cite{Wals} and references.  However, in the case that  the initial date is anticipating a new solution concept(integral) is need to define. In this direction, we only know  the work   \cite{Tindel} where the author considerer the finite dimensional initial date.

\noindent Let us considerer the following SPDE with homogeneous Dirichlet boundary conditions

\begin{equation}\label{para}
 \left \{
\begin{array}{lll}
u \ dt & = & \triangle u \ dt + g(t,x,u)  dW_{t}, \ t\geq 0 , \ 0 \leq x \leq 1  \\
u(t,0) & = & u(t,1)=0, \   \\
u_0 & = & f(x,F)
\end{array}
\right .
\end{equation}

 \noindent under the following hypothesis  :

$\mathbf{a)}$  $W_{t}$ is a  standard cylindrical Wiener process in  $L^{2}([0,1])$ .

$\mathbf{b)}$ $g$ is continuous  function and  satisfies

 \[
 |\  g(t,x,y_1)-g(t,x,y_2)|\leq \ C |y_1-y_2|,
 \]

$\mathbf{c)}$ $f(.,z) \in C_{0}([0,1])$ and  satisfies

\[
 |\  f(x,z_1)-f(x,z_2)|\leq \ C_{N} |z_1-z_2|,
 \]

\noindent for $ |z_1|, |z_2|\leq N, N>0 $.

$\mathbf{d)}$ $F$ be a d-dimensional $ \mathcal{F}_T$-measure r.v.

\begin{definition}\label{defisolu}
 A  mild   solution of the SPDE (\ref{para}) is a   stochastic process
    $u$ such that  it satisfies

\[
 u(t,x)  = \int_{0}^{1}  G(t-s,x,y) \  f(y,F) \ dy
+ \int_{0}^{t}  G(t-s,x,.) \ g(s,.,u)   \ dW_{s}^ {-} .
\]
\ where  $G(t,x,y)$ is   the fundamental solution of the heat equation with Dirchlet boundary conditions.

\end{definition}

\noindent  We recall that   $G(t,x,y)$ satisfies

\begin{equation}\label{inefunda}
\int_{0}^{1} G(t,x,y)^{p} dy \leq  C_{p} \ t^{\frac{1-p}{2}} \ for \ all \ p >0.
\end{equation}

\begin{remark}\label{lem1}(see \cite{Nual}) Suppose that $ \{ Y_{n}(z) : z\in \mathbb{R}^{d}, n\geq 1  \}$ is a sequence
of random field such that $Y_{n}(z)$ converge in probability to $Y(z)$ as $n$  tends to infinity, for each
$z\in \mathbb{R}^{d}$. Suppose that

\[
\mathbb{E} |Y(z_{1})- Y(z_{2})|^{p} \leq C_{N}   |z_1-z_2|^{\alpha} .
\]

\noindent  for $ |z_1|, |z_2|\leq N, N>0 $  and  for  some  constants   $ p>0 $  and  $\alpha>d$. Then, for any d-dimensional r.v F one has

\[
\lim_{n\rightarrow \infty}Y_{n}(F)=Y(F) .
\]

\noindent  in probability. Moreover, the convergence is $L^{p}(\Omega)$ if $F$ is bounded.
 \end{remark}

\begin{theorem}\label{teopara} Assume that $\mathbf{a)}$, $\mathbf{b)}$,  $\mathbf{c)}$ and $\mathbf{d)}$ hold . Then
there exists  a mild solution $u$  for the SPDE (\ref{para}).
\end{theorem}

\begin{proof}
{\large Step 1}(Auxiliary problems) We
considerer  the following family of   problems

\begin{equation}\label{Auxilia1}
 v^z(t,x)  = \int_{0}^{1}  G(t-s,x,y) \ f(y,z) \ dy
+ \int_{0}^{t}  G(t-s,x,.) \ g(s,.,v^z)  \ dW_{s}.
\end{equation}

\noindent We observe that  $z\in\mathbb{R}^{d}$ is a parameter. It is know that there exists an
unique  solution $v^z(t,x) $ for the  problem
(\ref{Auxilia1})  and  it verifies $\sup_{0\leq t \leq T, 0\leq x \leq 1} \mathbb{E}[ |v^z(t,x)|^{2}]$    (see for instance
p. 142 of \cite{Nual}).

{\large Step 2} (Estimative of $v^z(t,x)$ )  We claim that $v^z(t,x)$ verifies

\begin{equation}\label{esti}
\mathbb{E} | v^{z_{1}}(t,x)-v^{z_{2}}(t,x)   |^{2} \leq C_{N}   |z_1-z_2|^{2} .
\end{equation}

\noindent  for $ |z_1|, |z_2|\leq N, N>0$. In fact, we have

\[
\mathbb{E} | v^{z_{1}}(t,x)-v^{z_{2}}(t,x)   |^{2} \leq C \ (\sup_{y} \mathbb{E}|f(y,z_{1})-f(y,z_{2})| )^{2}
 \]

\[
  + \ C \int_{0}^{t} \int_{0}^{1} G^{2}(t-s,x,y) \  \mathbb{E}|g(t,y,v^{z_{1}})-  g(t,y,v^{z_{2}})|^{2} \   dyds  .
\]

\noindent By hypothesis $\mathbf{b)}$,  $\mathbf{c)}$ it follows

\[
\mathbb{E} | v^{z_{1}}(t,x)-v^{z_{2}}(t,x) |^{2} \leq C_{N} \ |z_{1}-z_{2}|^{2}
 \]

\[
  + \ C  \ \int_{0}^{t}  \sup_{y} \mathbb{E}|v^{{z_{1}}}(t,y)-  v^{{z_{2}}}(t,y)|^{2} \ \int_{0}^{1} G^{2}(t-s,x,y)  \  dyds  .
\]

\noindent Taking supremo  we obtain

\[
\sup_{x} \mathbb{E} | v^{z_{1}}(t,x)-v^{z_{2}}(t,x) |^{2} \leq C_{N} \ |z_{1}-z_{2}|^{2}
 \]

\[
  + \ C \   \int_{0}^{t}  \sup_{y} \mathbb{E}|v^{z_{1}}(t,y)-  v^{z_{2}}(t,y)|^{2} \ (t-s)^{-\frac{1}{2}} \ ds  .
\]

\noindent Finally by the Gronwall lemma we get the inequality (\ref{esti}).

{\large Step 3} (Our solution)
We shall show  that $u(t,x):= v^F(t,x)$ is a mild solution of the  problem (\ref{para}). Let  $\{v_j, j\in  \mathbb{N} \}$ be a complete orthonormal basis
 of $L^{2}([0,1])$.

 \noindent Combining  hypothesis  $\mathbf{b)}$, inequality (\ref{inefunda}) and step 2 we have that

\begin{equation}\label{esti2}
\mathbb{E} |<G(t-s,x,.) (g(s,.,v^{z_{1}}) -g(s,.,v^{z_{1}})),v_{j}>| ^{2}
 \leq C_{N}   |z_1-z_2|^{2} .
\end{equation}

\noindent and

\begin{equation}\label{esti3}
\mathbb{E} |\int_{0}^{t}  G(t-s,x,.) ( g(s,.,v^{z_{1}})- g(s,.,v^{z_{2}})  ) dW_{s}|^{2}
 \leq C_{N}   |z_1-z_2|^{2} .
\end{equation}

\noindent  for $ |z_1|, |z_2|\leq N, N>0 $.

 \noindent From inequality   (\ref{esti2})
 and  by the  Russo-Vallois substitution theorems(see for example theorem 1.1 of  \cite{RV1} ) we obtain

 \[
c_j =\int_{0}^{t} <G(t-s,x,.) g(s,.,u), v_{j}> dB_{s}^{-}(v_{j})=
\]

 \[
(\int_{0}^{t}   G(t-s,x,.) g(s,.,v^z), v_{j}> dB_{s}(v_{j}))(F).
\]

\noindent Moreover, from  remark  \ref{lem1} and inequality (\ref{esti3}) we have $\sum_{j=1}^{m} c_j$
converges in probability  and

\[
\lim_{m\rightarrow\infty} \sum_{j=1}^{m} c_j= \int_{0}^{t}  G(t-s,x,.) \ g(s,.,u)  \ dW_{s}^ {-}= (\int_{0}^{t}  G(t-s,x,.) \ g(s,.,v^{z}) \ dW_{s}) (F).
\]

\noindent Thus $u(t,x):= v^F(t,x)$ is a mild  solution of the  SPDE (\ref{para}).
\end{proof}

\noindent We end the paper indicating some future works.

\begin{remark} An interesting and potential extension : is the extension for the infinite dimensional stochastic integral in the setup
of Da Prato and Zabczyk \cite{DaPra}.  This infinite-dimensional stochastic integral can be written as a series of It\^o stochastic integrals, see for instance the recent presentation of \cite{Da}. We think that can be used  a similar scheme to the presented  in section 3.
 \end{remark}

\begin{remark} It is very well know the relation between
 the  Skorohod-It\^o integral and the Wick product   (see \cite{Russo}), this is

 \[
\int_{0}^{t} X_t dW_{s}=\int_{0}^{t} X_t \diamond W_{s} ds .
\]

 \noindent An interesting future work is to study this type  relation
 between integrals via regularization (Forward and symmetric) for the cylindrical   Wiener process
 and  distribution products defined via expansion in series(see \cite{calo},
  \cite{CO1} and  \cite{CO2}). Moreover,   we are interested in to  study   generalized solutions of SPDE  drive by a cylindrical   Wiener process  in a  Colombeau type algebra in the new  spirit  given in
\cite{CO1} and \cite{CO2}.
 \end{remark}

\end{document}